\documentclass{au}
\usepackage{graphicx}
\usepackage{newlattice}

\newcommand{\EB}{\E B}

\newtheorem{theorem}{Theorem}[section]

\theoremstyle{definition}

\newtheorem{example}[theorem]{Example}
\newtheorem{construction}[theorem]{Construction}

\begin{document}
\title{Multi-algebras as tolerance quotients of algebras}  
\author{G. Gr\"{a}tzer} 
\address{Department of Mathematics\\
  University of Manitoba\\
  Winnipeg, MB R3T 2N2\\
  Canada}
\email[G. Gr\"atzer]{gratzer@me.com}
\urladdr[G. Gr\"atzer]{http://server.maths.umanitoba.ca/homepages/gratzer/}

\author{R. Quackenbush}
\email[R. Quackenbush]{qbush@cc.umanitoba.ca}


\begin{abstract}
If $A$ is an algebra and \bgt is a tolerance on $A$, 
then $A/\bgt$ is a multi-algebra in a natural way. 
We give an example to show that not every multi-algebra 
arises in this manner. 
We slightly generalize the construction of $A/\bgt$ 
and prove that every multi-algebra arises from this modified construction.
\end{abstract}

\maketitle

\section{Tolerance quotients} \label{S:I}

Let $M$ be a nonempty set. 
Denote by $\Pow M$ the set of all subsets of $M$ 
and by $\Pow_+ M$ the set of all nonempty subsets of $M$.
A \emph{multi-operation} $f$ on $M$ (of~\emph{arity}~$n$) 
is a function $f \colon M^n \to \Pow_+ M$.
A \emph{multi-algebra} $(M; F)$ is a nonempty set $M$ 
with a set $F$ of multi-operations on $M$.

If $(A; F)$ is an algebra and \bga is a congruence on $A$, 
then the congruence classes of~\bga form an algebra 
$(A/\bga; F) = (A; F)/\bga$. 
A congruence \bga is a reflexive, symmetric, 
and transitive binary relation on $A$ 
with the Substitution Property 
(so $\bga$ is a subalgebra of $(A; F)^2$). 
If we drop the Substitution Property, 
$(A/\bga; F) = (A; F)/\bga$ becomes a multi-algebra. 
The converse was proved in G.~Gr\"atzer \cite{gG62}: 
every multi-algebra can be obtained (up to isomorphism)
in~ this fashion; 
for additional results in this direction, see 
G. Gr\"atzer and G.\,H. Wenzel \cite{GW89}, H. H\"oft and P.\,E. Howard~\cite{HH81}.

What happens if we drop the transitivity of \bgt? 
Define a \emph{tolerance} \bgt on an algebra $(A; F)$ as a reflexive and symmetric binary relation on $A$ with the Substitution Property. 
For an overview of tolerances in algebra, see I. Chajda \cite{iC91}. 

Let $\bgt$ be a binary relation on the set $A$. 
As in graph theory, we call a~subset $B \ci A$ a \emph{clique} 
if $B^2 \ci \bgt$; we call $B$ a \emph{maximal clique} if $B$ is maximal
with~respect to the property  $B^2 \ci \bgt$. By Zorn's Lemma, every clique is contained in a maximal clique. 
A~\emph{covering} of a nonempty set $A$ is a collection 
$\E C$ of~pairwise incomparable subsets of~$A$ whose union is $A$. 
It is easy to see that a reflexive and symmetric 
binary relation $\bgt$ on $A$ is equivalent to a covering, 
$\E C_{\bgt}$, of $A$ where the sets in the covering are 
the maximal cliques of \bgt. 
For~a~tolerance \bgt on an algebra $(A; F)$, 
we call a maximal clique a \emph{tolerance block}. 

Conversely, a covering $\E C$ of $A$ \emph{induces} 
a reflexive and symmetric binary relation~$\bgt$ on $A$ 
defined by $(a, b) \in \bgt$ if{f} there is an $S \in \E C$ 
such that $a, b \in S$. 
It is easy to see that each $S \in \E C$ is a clique of $\bgt$; 
however, $S \in \E C$ need not be a maximal clique of~$\bgt$, 
and even if each $S \in \E C$ is a maximal clique of $\bgt$, 
there may be maximal cliques of~$\bgt$ 
that do not belong to $\E C$.

\begin{example}\label{E:1}
Let $|A| \geq 3$. 

(1) Let $\E C_2$ be the set of all $2$-element subsets of $A$. Then $\E C_2$ induces the full relation $\bgt = A^2$ on $A$, and $A$ is its unique maximal clique.

(2) On $\Pow_+ A$ define the \emph{non-disjointness} relation $\bgn$ by
\begin{equation}\label{E:nondisjoint}
   (X, Y) \in \bgn \text{\q if{f}\q} X \ii Y \neq \es.
\end{equation} 
Then every ultrafilter $U$ on $A$ is a maximal clique of \bgn. But these are not the only maximal cliques of \bgn. Let $\set{a, b, c}$ be a $3$-element subset of $A$; then the set 
\[
   \set{\set{a, b}, \set{a, c}, \set{b, c}}
\]
is a clique of \bgn, but is contained in no ultrafilter of $A$. Hence, there is a maximal clique of \bgn which is not an ultrafilter. Notice that \bgn is induced by the set 
of~all principal ultrafilters of $A$. The non-disjointness relation is well-studied in the combinatorics of set systems, usually under the name \emph{intersecting families} of~sets; see \cite{bB86}.
\end{example}

\begin{construction}\label{C:1}
Given a tolerance $\bgt$ on the algebra $(A; F)$, 
we construct a multi-algebra $(M; F)$. 
Let $\EB_{\bgt}$ be the covering of $A$ 
consisting of all tolerance blocks of $\bgt$. 
For $f \in F$ of arity $n$ and 
$B_1, \dots, B_n \in \EB_{\bgt}$, define $f$ on $\EB_{\bgt}$ by
\[
   f(B_1, \dots, B_n) = 
   \setm{B \in \EB_{\bgt}}{f(b_1, \dots, b_n) \in B 
       \text{ for all } b_i \in B_i \text{ and } 1 \leq i \leq n}.
\]
Since \bgt is a tolerance, if $f(b_1, \dots, b_n) \in B$ 
for some $b_i \in B_i$, then $f(b_1, \dots, b_n) \in B$ 
for all $b_i \in B_i$. 
We denote by $(A; F)/\bgt$
this multi-algebra defined on $\EB_{\bgt}$ 
and call it the \emph{tolerance quotient} 
induced by the tolerance $\bgt$ on the algebra $(A; F)$.
\end{construction}

It is natural to wonder how general this construction is. 
As the next example shows, it is not completely general.

\begin{example}\label{E:2}
Let $M = \set{1, 2, 3}$ and define the  multi-groupoid $(M;+)$  by 
\[
   a + b = \set{a, b}
\]
for all $a, b \in \set{1, 2, 3}$. 
Let us assume that $(M;+)$ is isomorphic 
to the tolerance multi-groupoid induced by the tolerance $\bgt$ 
on the groupoid $(A; +)$. 
Then $\bgt$  has exactly 3 tolerance blocks: $B, C, D$. 
Further, $B + C = \set{B, C}$, $B + D = \set{B, D}$, 
and $C + D = \set{C, D}$. 
This  means that there are $b_1, b_2 \in B$, 
$c_1, c_3 \in C$, and $d_2, d_3 \in D$ 
such that $b_1 + c_1 \in (B \ii C) - D$, 
$b_2 + d_2 \in (B \ii D) - C$, and 
$c_3 + d_3 \in (C \ii D) - B$. 
But then $\set{b_1 + c_1, b_2 + d_2, c_3 + d_3}$ 
is a $3$-element clique in $\bgt$ not contained 
in $B$, $C$, or $D$, so  $\bgt$ has at least 4  cliques. 
This is a contradiction, 
so there is no tolerance multi-groupoid isomorphic 
to the multi-groupoid $(M;+)$.
\end{example}

\section{Full covering quotients}

Let us re-examine the multi-groupoid $(M;+)$ of Example~\ref{E:2}. 
We take $A = \Pow_+ M$ and define \bgn to be the 
non-disjointness relation as in \eqref{E:nondisjoint}. 
Then \bgn has 4 maximal cliques: 
\begin{align*}
S_1 &= \set{\set{1}, \set{1,2}, \set{1,3}, \set{1,2,3}},\\
S_2 &= \set{\set{2}, \set{1,2}, \set{2,3}, \set{1,2,3}},\\ 
S_3 &= \set{\set{3}, \set{1,3}, \set{2,3}, \set{1,2,3}},\\
S   &= \set{\set{1,2}, \set{1,3}, \set{2,3}, \set{1,2,3}}.
\end{align*}
Let $B, C \in A$. If $|B| = |C| = 1$, 
define $B + C$ to be $B \cup C$; 
otherwise, define $B + C$ to be $M = \set{1, 2, 3}$. 
Then \bgn is readily seen to be a tolerance on~$(A; +)$. 
Form $(A; +)/\bgn$. 
We hope that $(M;+)$ is isomorphic to the subalgebra 
of~ $(A; +)/\bgn$ on the subset $\set{S_1, S_2, S_3}$. 
But it is not: $S_1 + S_2 = \set{S_1, S_2, S}$, 
which is not a subset of $\set{S_1, S_2, S_3}$.

So we have to get rid of that troublesome tolerance block $S$. 

\begin{construction}\label{C:2}
Let $\bgt$ be a tolerance  on the algebra $(A; F)$.  
Let $\E C$ be a~covering  of $A$ consisting of some tolerance blocks of $\bgt$.
If for $a, b \in A$, 
we have $(a, b) \in \bgt$ 
if{f} there is some $C \in \E C$ such that $a, b \in C$,
then we call $\E C$ a \emph{full covering}. 
For $f \in F$ of arity $n$ and $C_1, \dots, C_n \in \E C$, define 
\[
   f(C_1, \dots, C_n) = \setm{C \in \E C}{f(c_1, \dots, c_n) \in C \text{ for all } c_i \in C_i \text{ and } 1 \leq i \leq n}.
\]
We denote this multi-algebra by $(A; F)/\E C$ 
and call it the \emph{full covering quotient} of the algebra $(A; F)$
induced by the full covering $\E C$.
\end{construction}

We shall now see that Construction \ref{C:2} is completely general.

\begin{theorem}\label{T:1}
Let $(M; F)$ be a multi-algebra. 
Then there is an algebra $(A; F)$
and a full covering $\E C$ of $(A; F)$ that induces 
a tolerance $\bgt$ on $(A; F)$ 
such that $(M; F)$ is isomorphic to $(A; F)/\E C$.
\end{theorem}

\begin{proof}
Take $A = \Pow_+ M$. For $n$-ary $f \in F$ and $m_1, \dots, m_n \in M$, 
define $f$ on $\Pow_+ M$ by:
\[
   f(\set{m_1}, \dots, \set{m_n})= \set{f(m_1, \dots, m_n)},
\]
and, otherwise, 
\[
   f(B_1, \dots, B_n) = M.
\]
For $m \in M$, let $M_m = \setm{S \ci M}{m \in S}$; 
then $\E C = \setm{M_m}{m \in M}$ is a~covering of $A$ by the principal filters of $M$. 
As noted in Example \ref{E:1} (2), $\E C$ induces \bgn on $A$. 
Each $M_m$ is a maximal clique of $\bgn$.

To prove that $\bgn$ is a tolerance, 
let $f \in F$ be $n$-ary and $(A_i, B_i) \in \bgn$ 
for $1 \leq i \leq n$; we need to conclude that 
$(f(A_1, \dots, A_n), f(B_1, \dots, B_n)) \in \bgn$. 
Note that $(M, N) \in \bgn$ 
for all nonempty $N \ci M$; 
thus, our conclusion holds unless each $A_i$ 
and each $B_i$ is a $1$-element subset of $M$. 
But for $1$-element subsets $A, B$, 
we have $(A, B) \in \bgn$ if{f} $A = B$.  
Because $\bgn$ is reflexive, 
our conclusion also holds in this case since then $A_i = B_i$ 
for all $i$. So $\bgn$ is indeed a tolerance. Hence, $\E C$ is a full covering.

Now use Construction \ref{C:2} to form $(A; F)/\E C$.  
It is obvious from the definitions 
that $(A; F)/\E C$ is isomorphic 
to $(M; F)$ via the map that sends $M_m$ to~$m$.
\end{proof}

\end{document}